\input ./style/arxiv-vmsta.cfg
\documentclass[numbers,compress,v1.0.1]{vmsta}

\volume{2}
\issue{3}
\pubyear{2015}
\firstpage{287}
\lastpage{295}
\doi{10.15559/15-VMSTA39CNF}

\newtheorem{thm}{Theorem}
\newtheorem{pro}{Proposition}

\theoremstyle{definition}
\newtheorem{rem}{Remark}

\startlocaldefs
\renewcommand{\d}{\mathrm{d}}
\newcommand{\s}{\mathbf{s}}
\renewcommand{\t}{\mathbf{t}}
\renewcommand{\u}{\mathbf{u}}
\newcommand{\R}{\mathbb{R}}

\allowdisplaybreaks
\endlocaldefs

\begin{document}
\begin{frontmatter}

\title{Fredholm representation of multiparameter Gaussian processes
with applications to equivalence in law and series expansions\tnotemark[t2]}

\author[a]{\inits{T.}\fnm{Tommi}\snm{Sottinen}\corref{cor1}}\email{tommi.sottinen@iki.fi}
\cortext[cor1]{Corresponding author.}
\author[b]{\inits{L.}\fnm{Lauri}\snm{Viitasaari}\fnref{f1}}\email{lauri.viitasaari@aalto.fi}
\fntext[f1]{Lauri Viitasaari was partially funded by Emil Aaltonen Foundation.}

\address[a]{Department of Mathematics and Statistics, University of
Vaasa, P.O. Box 700, FIN-65101 Vaasa, Finland}
\address[b]{Department of Mathematics and System Analysis, Aalto
University School of Science, P.O. Box 11100, FIN-00076 Aalto, Finland}

\tnotetext[t2]{\xch{The authors thank the}{We thank the} referees for their useful comments.}

\markboth{T. Sottinen, L. Viitasaari}{Fredholm representation}

\begin{abstract}
We show that every multiparameter Gaussian process with integrable
variance function admits a Wiener integral representation of Fredholm
type with respect to the Brownian sheet. The Fredholm kernel in the
representation can be constructed as the unique symmetric square root
of the covariance. We analyze the equivalence of multiparameter
Gaussian processes by using the Fredholm representation and show how to
construct series expansions for multiparameter Gaussian processes by
using the Fredholm kernel.
\end{abstract}


\received{25 June 2015}
\revised{21 September 2015}
\accepted{21 September 2015}
\publishedonline{2 October 2015}
\end{frontmatter}

\section{Introduction}

In this article, we consider multiparameter processes, that is, our
time is multidimensional. Throughout the paper, the dimension of time
$n\ge1$ is arbitrary but fixed.

We use the following notation throughout this article: $\t,\s,\u\in
\R^n$ are $n$-dimen\-sional multiparameters of time: $\t=(t_1,\ldots,
t_n)$, $\s=(s_1,\ldots,s_n)$, $\u=(u_1,\ldots,u_n$); $\mathbf{0}$
is an $n$-dimensional vector of $0$s, and $\mathbf{1}$ is an
$n$-dimensional vector of $1$s. We denote $\s\le\t$ if $s_k\le t_k$
for all $k\le n$. For $\s\le\t$, the set $[\s,\t]\subset\R^n$ is
the $n$-dimensional rectangle $\{ \u\in\R^n; \s\le\u\le\t\}$.

Let $X=(X_\t)_{\t\in[\mathbf{0},\mathbf{1}]}$ be a real-valued
centered Gaussian multiparameter process or field defined on some
complete probability space $(\varOmega,\mathcal{F},\mathbb{P})$. We
assume that the Gaussian field $X$ is \emph{separable}, that is, its
linear space, or the 1st chaos,
\[
\mathcal{H}_1 = \mathrm{cl}\,\, \bigl(\mathrm{span} \bigl\{
X_\t\, ;\, \t\in[\mathbf{0},\mathbf{1}] \bigr\} \bigr)
\]
is separable. Here $\mathrm{cl}$ means closure in $L^2(\varOmega
,\mathcal{F},\mathbb{P})$.

Our main result, Theorem~\ref{thm:fredholm}, shows when the Gaussian
field $X$ can be represented in terms of the Brownian sheet. Recall
that the Brownian sheet $W=(W_\t)_{\t\in[\mathbf{0},\mathbf{1}]}$
is the centered Gaussian field with the covariance
\[
\mathbb{E}\! [ W_\t W_\s ] = \prod
_{k=1}^n \min(t_k,s_k).
\]
The Brownian sheet can also be considered as the \emph{Gaussian white
noise} on $[\mathbf{0},\mathbf{1}]$ with the Lebesgue control
measure. This means that $\d W$ is a random measure on $ ([\mathbf
{0},\mathbf{1}],\mathcal{B}([\mathbf{0},\mathbf{1}]),\mathrm
{Leb}([\mathbf{0},\mathbf{1}]) )$ characterized by the
following properties:
\begin{enumerate}
\item$\int_A \d W_\t\sim\mathcal{N}(0,\mathrm{Leb}(A))$,
\item$\int_A \d W_\t$ and $\int_B \d W_\s$ are independent if
$A\cap B = \varnothing$.
\end{enumerate}
Then, if $f,g:[\mathbf{0},\mathbf{1}] \to\mathbb{R}$ are simple
functions, then we have the \emph{Wiener--It\^o isometry}
\begin{equation}
\label{eq:isometry} \mathbb{E}\! \biggl[\int_{[\mathbf{0},\mathbf{1}]} f(\t)\, \d
W_\t \int_{[\mathbf{0},\mathbf{1}]} g(\s)\, \d W_\s
\biggr] = \int_{[\mathbf{0},\mathbf{1}]} f(\t)g(\t)\, \d\t.
\end{equation}
Consequently, the integral $\int_{[\mathbf{0},\mathbf{1}]} f(\t)\,\d
W_\t$ can be extended for all $f\in L^2([\mathbf{0},\mathbf{1}])$ by
using the isometry \eqref{eq:isometry}, and the isometry \eqref
{eq:isometry} will also hold for this extended integral.

In this article, we show the Fredholm representation for Gaussian
fields satisfying the trace condition \eqref{eq:trace} in Section~\ref
{sect:fredholm}, Theorem~\ref{thm:fredholm}. In Section~\ref
{sect:equivalence}, we apply the Fredholm representation to give a
representation for Gaussian fields that are equivalent in law, and in
Section~\ref{sect:series}, we show how to generate series expansions
for Gaussian fields by using the Fredholm representation. The Fredholm
representation of Theorem~\ref{thm:fredholm} can also be used to
provide a \emph{transfer principle} that builds stochastic analysis
and Malliavin calculus for Gaussian fields from the corresponding
well-known theory for the Brownian sheet. We do not do that in this
article, although it would be quite straightforward given the results
for the one-dimensional case provided in \cite
{Sottinen-Viitasaari-2014-preprint}.

\section{Fredholm representation}\label{sect:fredholm}

Recall that $X$ is a separable centered Gaussian field with covariance
function~$R$ and $W$ is a Brownian sheet. Suppose that $X$ is defined
on a complete probability space $(\varOmega,\mathcal{F},\mathbb{P})$
that is rich enough to carry Brownian sheets.

The following theorem states that the field $X$ can be realized as a
Wiener integral with respect to a Brownian sheet.
Let us note that it is not always possible to construct the Brownian
sheet $W$ directly from the field $X$. Indeed, consider the trivial
field \mbox{$X\equiv0$} to see this. As a consequence, the Karhunen
representation theorem (see, e.g., \cite
[Thm.~41]{Berlinet-Agnan-2004}) cannot be applied here. Consequently,
the Brownian sheet in representation \eqref{eq:fredholm} is not
guaranteed to exist on the same probability space with $X$.

In any case, representation \eqref{eq:fredholm} holds in law.
This means that for a given Brownian sheet $W$, the field given by
\eqref{eq:fredholm} is a Gaussian field with the same law as $X$.

\begin{thm}[Fredholm representation]\label{thm:fredholm}
Let $(\varOmega,\mathcal{F},\mathbb{P})$ be a probability space such
that $\sigma\{\xi_k;k\in\mathbb{N}\}\subset\mathcal{F}$, where
$\xi_k$, $k\in\mathbb{N}$, are i.i.d.\ standard normal random variables.
Let $X$ be a separable centered Gaussian field defined on $(\varOmega
,\mathcal{F},\mathbb{P})$.
Let $R$ be the covariance of $X$.

Then there exist a kernel $K\in L^2([\mathbf{0},\mathbf{1}])$ and a
Brownian sheet $W$, possibly, defined on a larger probability space,
such that the representation
\begin{equation}
\label{eq:fredholm} X_\t= \int_{[\mathbf{0},\mathbf{1}]} K(\t,\s)\, \d
W_\s
\end{equation}
holds \emph{if and only if} $R$ satisfies the trace condition
\begin{equation}
\label{eq:trace} \int_{[\mathbf{0},\mathbf{1}]} R(\t,\t)\, \d\t< \infty.
\end{equation}
\end{thm}

\begin{proof}
From condition \eqref{eq:trace} it follows that the covariance operator
\[
\mathrm{R} f(\t) = \int_{[\mathbf{0},\mathbf{1}]} f(\s) R(\t,\s )\, \d\s
\]
is Hilbert--Schmidt. Indeed, the Hilbert--Schmidt norm of the operator
$\mathrm{R}$ satisfies, by the Cauchy--Schwarz inequality,
\begin{align*}
{\|\mathrm{R}\|}_{\mathrm{HS}} &= \sqrt{\int_{[\mathbf{0},\mathbf{1}]}\int
_{[\mathbf{0},\mathbf
{1}]} R(\t,\s)^2\, \d\t\,\d\s}
\\
&\le \sqrt{\int_{[\mathbf{0},\mathbf{1}]}\int_{[\mathbf{0},\mathbf
{1}]} R(\t,\t)
R(\s,\s)\, \d\t\,\d\s}
\\
&= \int_{[\mathbf{0},\mathbf{1}]} R(\t,\t)\, \d\t.
\end{align*}
Since Hilbert--Schmidt operators are compact operators, it follows
from, for example, \cite[p. 233]{Riesz-SzNagy-1955} that the operator
$\mathrm{R}$ admits the eigenfunction representation
\begin{equation}
\label{eq:eig-rep} \mathrm{R} f(\t) = \sum_{k=1}^\infty
\lambda_k \int_{[\mathbf{0},\mathbf{1}]} f(\s) \phi_k(\s)\,\d\s\,\, \phi_k(\t).
\end{equation}
Here $(\phi_k)_{k=1}^\infty$, the eigenfunctions of $\mathrm{R}$,
form an orthonormal system on $L^2([\mathbf{0},\mathbf{1}])$. In
particular, this means that
\begin{equation}
\label{eq:eig-kernel} R(\t,\s) = \sum_{k=1}^\infty
\lambda_k \,\phi_k(\t)\phi_k(\s).
\end{equation}
From this it follows that the square root of the covariance operator
$\mathrm{R}$ admits a kernel $K$ if and only if
\begin{equation}
\label{eq:trace-lambda} \sum_{k=1}^\infty
\lambda_k < \infty.
\end{equation}
Note that condition \eqref{eq:trace-lambda} is equivalent to condition
\eqref{eq:trace}. Consequently, we can define
\begin{equation}
\label{eq:K-series} K(\t,\s) = \sum_{k=1}^\infty
\sqrt{\lambda_k}\,\phi_k(\t)\phi _k(\s)
\end{equation}
since the series in the right-hand side of \eqref{eq:K-series}
converges in $L^2([\mathbf{0},\mathbf{1}])$, and the eigenvalues
$ (\lambda_k )_{k=1}^\infty$ of a positive-definite
operator $\mathrm{R}$ are nonnegative.

Now,
\begin{align*}
R(\t,\s) &= \sum_{k=1}^\infty
\lambda_k \,\phi_k(\t)\phi_k(\s)
\\
&= \sum_{k=1}^\infty\sum
_{\ell=1}^\infty\sqrt{\lambda_k}\sqrt {
\lambda_\ell} \phi_k(\t)\phi_\ell(\s) \int
_{[\mathbf
{0},\mathbf{1}]} \phi_k(\u)\phi_\ell(\u)\, \d\u
\\
&= \int_{[\mathbf{0},\mathbf{1}]} \Biggl(\sum_{k=1}^\infty
\sqrt {\lambda_k}\phi_k(\t)\phi_k(\u) \, \sum
_{\ell=1}^\infty\sqrt{\lambda_{\ell}}
\phi_{\ell}(\s)\phi _{\ell}(\u) \Biggr) \d\u
\\
&= \int_{[\mathbf{0},\mathbf{1}]} K(\t,\u)K(\s,\u)\, \d\u,
\end{align*}
where the interchange of summation and integration is justified by the
fact that series \eqref{eq:K-series} converges in $L^2([\mathbf{
0},\mathbf{1}])$.
From this calculation and from the Wiener--It\^o isometry \eqref
{eq:isometry} of the integrals with respect to the Brownian sheet it
follows that the centered Gaussian processes on the left-hand side and
the right-hand side of Eq.~\eqref{eq:fredholm} have the same
covariance function. Consequently, since they are Gaussian fields, they
have the same law. This means that representation~\eqref{eq:fredholm}
holds in law.

Finally, we need to construct a Brownian sheet $W$ associated with the
field~$X$ such that representation \eqref{eq:fredholm} holds in
$L^2(\varOmega,\mathcal{F},\mathbb{P})$.
Let $(\tilde\phi_k)_{k=1}^\infty$ be any orthonormal basis on
$L^2([\mathbf{0},\mathbf{1}])$. Set
\[
\phi_k(\t) = \int_{[\mathbf{0},\mathbf{1}]} \tilde\phi_k(
\s) K(\t,\s)\, \d\s.
\]
Then $(\phi_k)_{k=1}^\infty$ is an orthonormal basis (possibly finite
or even empty!) on the reproducing kernel Hilbert space (RKHS) of the
Gaussian field $X$ (see further for a definition). Let $\varTheta$ be an
isometry from the RKHS to $L^2(\varOmega,\sigma(X),\mathbb{P})$. Set
$\xi_k = \varTheta(\phi_k)$. Then $\xi_k$ are i.i.d. standard normal
random variables, and by the reproducing property we have that
\[
X_\t= \sum_{k=1}^\infty
\phi_k(\t)\, \xi_k
\]
in $L^2(\varOmega,\mathcal{F},\mathbb{P})$. Now, it may be that there
are only finitely many $\xi_k$ developed this way. If this is the
case, then we augment the finite sequence $(\xi_k)_{k=1}^n$ with
independent standard normal random variables. Then set
\[
W_\t= \sum_{k=1}^\infty\int
_{[\mathbf{0},\t]}\tilde\phi_k(\s )\, \d\s\,\,
\xi_k.
\]
For this Brownian sheet, representation \eqref{eq:fredholm} holds in
$L^2(\varOmega,\mathcal{F},\mathbb{P})$. Indeed,
\begin{align*}
\int_{[\mathbf{0},\mathbf{1}]} K(\t,\s)\, \d W_\s &= \int
_{[\mathbf{0},\mathbf{1}]} K(\t,\s)\, \d\sum_{k=1}^\infty
\int_{[\mathbf{0},\t]}\tilde\phi_k(\s)\, \d\s\,\,
\xi_k
\\
&= \sum_{k=1}^\infty\int_{[\mathbf{0},\mathbf{1}]}
K(\t,\s)\tilde \phi_k(\s)\, \d\s\,\, \xi_k
\\
&= \sum_{k=1}^\infty\phi_k(\t)\,
\xi_k
\\
&= X_\t.
\end{align*}
Here the change of summation, differentiation, and integration is
justified by the fact that the everything is square integrable.
\end{proof}

\begin{rem}
\begin{enumerate}
\item
The eigenfunction expansion \eqref{eq:eig-kernel} for the kernel $(\t
,\s)\mapsto K(\t,\s)$ is symmetric in $\t$ and $\s$. Consequently,
it is always possible to have a symmetric kernel in representation
\eqref{eq:fredholm}, that is, in principle it is always possible to
transfer from a \emph{given} representation
\[
X_\t= \int_{[\mathbf{0},\mathbf{1}]} K(\t,\s)\, \d W_\s
\]
to
\[
X_\t= \int_{[\mathbf{0},\mathbf{1}]} \tilde K(\t,\s)\, \d\tilde
W_\s
\]
where $\tilde W$ is some other Brownian sheet, and the kernel $\tilde
K$ is symmetric. Unfortunately, for a given kernel $K$ and Brownian
sheet $W$, the authors do not know how to do this analytically.

\item
In general, it is not possible to choose a Volterra kernel $K$ in
\eqref{eq:fredholm}. By a~Volterra kernel we mean a kernel that
satisfies $K(\t,\s) =0$ if $s_k>t_k$ for some $k$.
To see why a Volterra representation is not always possible, consider
the following simple counterexample: $X_\t\equiv\xi$, where $\xi$
is a standard normal random variable. This field cannot have a Volterra
representation since Volterra fields vanish in the origin. A Fredholm
representation for this field is simply $X_\t= \int_{[\mathbf
{0},\mathbf{1}]} \d W_\s$ {\rm(}with suitable Brownian sheet $W$
depending on $\xi)$.

For a more complicated counterexample \textup{(}with $X_0=0)$ see
\textup{\cite[Example~3.2]{Sottinen-Viitasaari-2014-preprint}}.

Consequently, in general, it is not possible to generate \xch{a Gaussian}{za Gaussian}
field~$X$ on the rectangle $[\mathbf{0},\t]$ from the noise $W$ on
the same rectangle $[\mathbf{0},\t]$. Instead, the whole information
on the cube $[\mathbf{0},\mathbf{1}]$ may be needed.

\item
If the family $\{ K(\t,\,\cdot\,)\, ; \t\in[\mathbf{0},\mathbf
{1}]\}$ is total in $L^2([\mathbf{0},\mathbf{1}])$, then a Brownian
sheet in representation \eqref{eq:fredholm} exists on the same
probability space $(\varOmega,\mathcal F,\mathbb P)$. Moreover, in this
case, it can be constructed from the Gaussian field $X$. Indeed, in
this case, we can apply the Karhunen representation theorem \textup
{\cite[{Thm.~41}]{Berlinet-Agnan-2004}}.
\end{enumerate}
\end{rem}

The \emph{reproducing kernel Hilbert space} (RKHS) of the Gaussian
field $X$ is the Hilbert space $\mathcal{H}$ that is isometric to the
linear space $\mathcal{H}_1$, and the defining isometry is $R(\t
,\cdot)\mapsto X_\t$. In other words, the RKHS is the Hilbert space
of functions over $[\mathbf{0},\mathbf{1}]$ extended and closed
linearly by the relation
\[
{ \bigl\langle R(\t,\cdot),R(\s,\cdot) \bigr\rangle}_{\mathcal
{H}} = R(\t,\s).
\]
The RKHS is of paramount importance in the analysis of Gaussian
processes. In this respect, the Fredholm representation \eqref
{eq:fredholm} is also very important. Indeed, if the kernel $K$ of
Theorem~\ref{thm:fredholm} is known, then the RKHS is also known as
the following reformulation of Lifshits \cite
[Prop.~4.1]{Lifshits-2012} states.

\begin{pro}\label{pro:RKHS}
Let $X$ admit representation \eqref{eq:fredholm}. Then
\[
\mathcal{H} = \biggl\{ f\,;\, f(\t) = \int_{[\mathbf{0},\mathbf
{1}]} \tilde f(
\s)K(\t,\s)\,\d\s, \tilde f\in L^2\bigl([\mathbf {0},\mathbf{1}]\bigr)
\biggr\}.
\]
Moreover, the inner product in $\mathcal{H}$ is given by
\[
{ \langle f, g \rangle}_{\mathcal{H}} = \inf_{\tilde
f,\tilde g}\, \int
_{[\mathbf{0},\mathbf{1}]} \tilde f(\t) \tilde g(\t)\, \d\t,
\]
where the infimum is taken over all such $\tilde f$ and $\tilde g$ that
\begin{align*}
f(\t) &= \int_{[\mathbf{0},\mathbf{1}]} \tilde f(\s)K(\t,\s)\, \d\s,
\\
g(\t) &= \int_{[\mathbf{0},\mathbf{1}]} \tilde g(\t)K(\t,\s)\, \d\s.
\end{align*}
\end{pro}

\section{Application to equivalence in law}\label{sect:equivalence}

Two random objects $\xi$ and $\zeta$ are \emph{equivalent in law}
if, their distributions satisfy $\mathbb{P}[\xi\in B]>0$ if and only
if $\mathbb{P}[\zeta\in B]>0$ for all measurable sets $B$. On the
contrary, the random objects $\xi$ and $\zeta$ are \emph{singular in
law} if there exists a measurable set $B$ such that $\mathbb{P}[\xi
\in B]=1$ but $\mathbb{P}[\zeta\in B]=0$. For \emph{centered}
Gaussian random objects there is the well-known dichotomy that two
centered Gaussian objects are either equivalent or singular in law; see
\cite[Thm.~6.1]{Hida-Hitsuda-1993}.

There is a complete characterization of the equivalence by any two
Gaussian processes due to Kallianpur and Oodaira; see \cite
[Thms.~9.2.1 and 9.2.2]{Kallianpur-1980}. It is possible to extend this
to Gaussian fields and formulate it in terms of the operator $\mathrm
{K}$. The result would remain quite abstract, though. Therefore, we due
not pursue in that direction. Instead, the following Proposition~\ref
{pro:equivalence} gives a~partial solution to the problem what do
Gaussian fields equivalent to a given Gaussian field $X$ look like.
Proposition~\ref{pro:equivalence} uses only the Hitsuda representation
theorem, which is, unlike the Kallianpur--Oodaira theorem, quite concrete.

Let $\tilde X=(\tilde X_\t)_{\t\in[\mathbf{0},\mathbf{1}]}$ be a
centered Gaussian field with covariance function~$\tilde R$, and let
$X=(X_\t)_{\t\in[\mathbf{0},\mathbf{1}]}$ be a centered Gaussian
field with covariance function $R$.

\begin{pro}[Representation of equivalent Gaussian fields]\label
{pro:equivalence}
Suppose that $X$ has representation \eqref{eq:fredholm} with kernel
$K$ and Brownian sheet $W$. If
\begin{equation}
\label{eq:equivalence} \tilde X_\t= \int_{[\mathbf{0},\mathbf{1}]} K(\t,\s)\,
\d W_\s- \int_{[\mathbf{0},\mathbf{1}]}\int_{[\mathbf{s},\mathbf{1}]} K(
\t ,\s)L(\s,\u)\,\d W_\u\,\d\s
\end{equation}
for some $L\in L^2([\mathbf{0},\mathbf{1}])$, then $\tilde X$ is
equivalent in law to $X$.
\end{pro}

\begin{proof}
By \cite[Prop.~4.2]{Sottinen-Tudor-2006} we have the following
multiparameter version of the Hitsuda representation theorem: A
centered Gaussian field $\tilde W=(\tilde W_\t)_{\t\in[\mathbf
{0},\mathbf{1}]}$ is equivalent in law to a Brownian sheet if and only
if it admits the representation
\begin{equation}
\label{eq:hitsuda} \tilde W_\t= W_\t- \int
_{[\mathbf{0},\mathbf{t}]}\int_{[\mathbf
{0},\mathbf{s}]} L(\s,\u)\, \d
W_\u\,\d\s
\end{equation}
for some Volterra kernel $L\in L^2([\mathbf{0},\mathbf{1}])$.

Let then $X$ have the Fredholm representation
\begin{equation}
\label{eq:X-fredholm} X_\t= \int_{[\mathbf{0},\mathbf{1}]} K(\t,\s)\, \d
W_\s.
\end{equation}
Then $\tilde X$ is equivalent to $X$ if it admits the representation
\begin{equation}
\label{eq:tildeX-equivalence} \tilde X_\t= \int_{[\mathbf{0},\mathbf{1}]} K(\t,\s)\,
\d\tilde W_\s,
\end{equation}
where $\tilde W$ is related to $W$ by \eqref{eq:hitsuda}. But
Eq.~\eqref{eq:equivalence} implies precisely this.
\end{proof}

\begin{rem}
On the kernel level, Eq.~\eqref{eq:equivalence} states that
\[
\tilde K(\t,\s) = K(\t,\s) - \int_{[\s,\mathbf{1}]} K(\t,\u) L(\u,\s)\,\d
\u
\]
for some Volterra kernel $L\in L^2([\mathbf{0},\mathbf{1}])$.
\end{rem}

\section{Application to series expansions}\label{sect:series}

The Mercer square root \eqref{eq:K-series} can be used to build the
Karhunen--Lo\`eve expansion for the Gaussian process $X$. But the
Mercer form \eqref{eq:K-series} is seldom known. However, if we can
somehow find \emph{any} kernel $K$ such that representation \eqref
{eq:fredholm} holds, then we can construct a series expansion for $X$
by using the Fredholm representation of Theorem~\ref{thm:fredholm} as follows.

\begin{pro}[Series expansion]\label{pro:series-expansion}
Let $X$ be a separable Gaussian process with representation \eqref
{eq:fredholm}, and let $(\phi_k)_{k=1}^\infty$ be \emph{any}
orthonormal basis on $L^2([\mathbf{0},\mathbf{1}])$. Then $X$ admits
the series expansion
\begin{equation}
\label{eq:series-expansion} X_\t= \sum_{k=1}^\infty
\int_{[\mathbf{0},\mathbf{1}]} \phi_k(\s )K(\t,\s)\, \d\s\cdot
\xi_k,
\end{equation}
where the $(\xi_k)_{k=1}^\infty$ is a sequence of independent
standard normal random variables. The series \eqref
{eq:series-expansion} converges in $L^2(\varOmega,\mathcal{F},\mathbb
{P})$ and also almost surely uniformly if and only if $X$ is continuous.
\end{pro}

The proof below uses reproducing kernel Hilbert space technique. For
more details on this, we refer to \cite{Gilsing-Sottinen-2003}, where
the series expansion is constructed for fractional Brownian motion by
using the transfer principle.

\begin{proof}
The Fredholm representation \eqref{eq:fredholm} implies immediately
that the reproducing kernel Hilbert space of $X$ is the image $\mathrm
{K} L^2([\mathbf{0},\mathbf{1}])$ and $\mathrm{K}$ is actually an
isometry from $L^2([\mathbf{0},\mathbf{1}])$ to the reproducing
kernel Hilbert space of $X$. Indeed, this is what Proposition~\ref
{pro:RKHS} states.

The $L^2$-expansion \eqref{eq:series-expansion} follows from this due
to \cite[Thm.~3.7]{Adler-1990} and the equivalence of almost sure
convergence of \eqref{eq:series-expansion}, and the continuity of $X$
follows from \cite[Thm.~3.8]{Adler-1990}.
\end{proof}


\end{document}